\documentclass[12pt,a4paper,draft]{article}

\usepackage{amsmath,amsfonts,amsthm,amssymb}
\usepackage{mathrsfs}
\usepackage{authblk}
\newtheorem{theorem}{Theorem} 
\newtheorem{lemma}{Lemma}
\newtheorem{proposition}{Proposition}
\newtheorem{corollary}{Corollary}

\let\scr\mathscr

\def\Pb{\mathbf{P}}

\def\Ex{\mathbf{E}}
\def\BB{\mathbb{B}}

\def\PP{\mathbb{P}}

\def\1{\mbox{1\hspace{-.25em}I}}

\begin{document}
\title{Approximation of BSDE with Hidden Forward Equation and Unknown
  Volatility  }
\author[1]{O.V. Chernoyarov}
\author[2]{Yu.A. Kutoyants}

\affil[1]{\small  National Research University ``MPEI'', Moscow, Russia}
\affil[2]{\small  Le Mans University,  Le Mans,  France}
\affil[1]{\small Maikop State Technological University, Maikop,  Russia}
\affil[1,2]{\small  Tomsk State University, Tomsk, Russia}

\date{}

\maketitle
\begin{abstract}
In the present paper the problem of approximating the solution of BSDE is considered in the
case where the solution of forward equation is observed in the presence
of small Gaussian noise. We suppose that the volatility of the forward
equation depends on an unknown parameter. This approximation is made in
several steps. First we obtain a preliminary estimator of the unknown parameter,
then using Kalman-Bucy filtration equations and Fisher-score device we
construct an one-step MLE-process of this parameter. The solution of BSDE is
approximated by means of the solution of PDE and the One-step MLE-process. The
error of approximation is  described in different metrics. 

\end{abstract}


\section{Introduction}

Backward stochastic differential equations (BSDE) were first introduced in the
linear case by Bismuth \cite{B73}. The general (nonlinear) case was initiated
by Pardoux and Peng \cite{PP90}. Since then the BSDE attract high attention
and are intensively developed due to their importance in financial mathematics
and insurance (see, e.g. El Karoui {\it et al.}  \cite{EPQ97}, Ma and Yong
\cite{MY}, Shen and Wei \cite{SW16}, Sun et al. \cite{SZY20} and the
references therein).

Let us recall what is the BSDE in the Markovian case following \cite{EPQ97}. For
the sake of simplicity only one-dimensional processes are considered. Let us
consider a filtered probability space $\left(\Omega 
,\left({\cal F}_t\right)_{t\in \left[0,T\right]}, \PP\right) $ with the
filtration $\left({\cal F}_t\right)_{t\in \left[0,T\right]} $ satisfying the
{\it usual conditions}. Define the stochastic differential equation (called
{\it forward}) by
\begin{align}
\label{1-1}
 {\rm d} X_t=S(t,X_t)\;{\rm d} t+ \sigma  (t,X_t)\;{\rm d} W_t,\ \ X_0,\
0\leq t\leq T,
\end{align}
where $W_t,{\cal F}_t, 0\leq t\leq T$ is the standard Wiener process and $X_0$ is
${\cal F}_0$ measurable  initial value, respectively.   The  trend coefficient $S\left(t,x\right)$
and the diffusion coefficient $\sigma \left(t,x\right)^2$ satisfy the Lipschitz and
linear growth conditions
\begin{align}
\label{1-2}
&\left|S\left(t,x\right)-S\left(t,y\right)\right|+\left|\sigma
\left(t,x\right)-\sigma \left(t,y\right)\right|\leq  L\left|x-y\right|,\\
&\left|S\left(t,x\right)\right|+\left|\sigma
\left(t,x\right)\right|\leq  C\left(1+\left|x\right|\right),
\label{1-3}
\end{align}
where $L>0$ and $C>0$ are constants. Under these conditions the stochastic
differential equation has a unique strong solution (see Liptser and Shiryaev
\cite{LS01}).

The main problem is the following: {\it  Given two functions
  $F\left(t,x,z,\sigma \right)$   and $\Phi
  \left(x\right)$
 we have to construct two processes
  $\left(Z_t,\Sigma _t,{\cal F}_t, 0\leq t\leq T \right)$ such that the
  solution of the stochastic differential equation}
\begin{align}
\label{1-4}
   {\rm d} Z_t=-F(t,X_t,Z_t,\Sigma _t)\;{\rm d} t+\Sigma_t\;{\rm d} W_t,\ \;
   \; 0\leq t\leq T, 
\end{align}
{\it (called {\it backward}) has the terminal value }$Z_T=\Phi
\left(X_T\right)$. 

This equation is often written in integral form as follows
\begin{align*}
Z_t=\Phi
\left(X_T\right)+\int_{t}^{T}F(s,X_s,Z_s,\Sigma_s)\;{\rm d}s- \int_{t}^{T}\Sigma_s\,{\rm
  d}W_s,\quad 0\leq t\leq T .
\end{align*}
We suppose that the functions $F\left(t,x,y,z\right)$ and $\Phi \left(x\right)$
satisfy the conditions
\begin{align}
\label{1-5}
&\left|F\left(t,x,z_1,\sigma _1\right)-F\left(t,x,z_2,\sigma _2\right) \right|\leq
L\left(\left|z_1-z_2\right|+\left|\sigma_1-\sigma_2\right|\right) ,\\
\label{1-6}
&\left|F\left(t,x,z,\sigma \right) \right|+\left|\Phi \left(x\right)\right|\leq
C\left(1+\left|x\right|^p\right) ,
\end{align}
where $p\geq {1}/{2}$.

This is the so-called {\it Markovian case}. For the existence and uniqueness of
the solution see Pardoux and Peng \cite{PP92}.

The solution $\left(Y_t,Z_t,{\cal F}_t, 0\leq t\leq T\right)$ could be
constructed as follows.  Suppose that $u\left(t,x\right)$ is a solution of the
partial differential equation
\begin{equation*}
  \frac{\partial u}{\partial t}+S\left(t,x\right)\frac{\partial u}{\partial
    x}+\frac{\sigma \left(t,x\right)^2}{2} \frac{\partial^2 u}{\partial
    x^2}=-F\left(t,x,u, \sigma \left(t,x\right)\frac{\partial u}{\partial
    x}\right) ,
\end{equation*}
with the terminal condition  $u\left(T,x\right)=\Phi \left(x\right)  $. 

Let us set $Z_t=u\left(t,X_t\right),$ then by It\^o's formula we obtain 
\begin{align*}
{{\rm d}}Z_t&=\left[\frac{\partial u\left(t,X_t\right)}{\partial
    t}+S\left(t,X_t\right)\frac{\partial u\left(t,X_t\right)}{\partial
    x}+\frac{\sigma  
  \left(t,X_t\right)^2}{2} \frac{\partial^2 u\left(t,X_t\right)}{\partial x^2}
  \right]\,{{\rm d}}t\\
&\qquad \qquad  + \sigma \left(t,X_t\right)\frac{\partial u\left(t,X_t\right)}{\partial
  x}\,{{\rm d}}W_t,\qquad Y_0=u\left(0,X_0\right). 
\end{align*}
 We use the notation
\begin{align*}
\frac{\partial u\left(t,X_t\right)}{\partial
  x}=u_x'\left(t,X_t\right)=\left.\frac{\partial u\left(t,x\right)}{\partial
  x}\right|_{x=X_t}.
\end{align*}
Hence if we denote $\Sigma _t=\sigma  \left(t,X_t\right) u_x'\left(t,X_t\right)$, then this
equation becomes 
\begin{align*}
{{\rm d}}Z_t=-F\left(t,X_t,Z_t,\Sigma_t\right){{\rm d}}t+\Sigma_t\,{{\rm d}}W_t,\quad \qquad
Z_0=u\left(0,X_0\right),
\end{align*}
and $Z_T=u\left(T,X_T\right)=\Phi \left(X_T\right)$.
 Therefore the problem is solved and the  equation \eqref{1-4} is obtained with
 given terminal value.

 We are interested in the problem of the approximation of the solution
 $\left(Z_t,\Sigma _t,{\cal F}_t, 0\leq t\leq T\right)$ of BSDE in the
 case where the forward equation \eqref{1-1} contains an unknown
 finite-dimensional parameter $\vartheta $:
\begin{equation*}
 {\rm d} X_t=S(\vartheta ,t,X_t)\;{\rm d} t+\sigma  (\vartheta ,t,X_t)\;{\rm d} W_t,\ \ X_0,\
0\leq t\leq T.
\end{equation*}
Then the solution $u$ of the corresponding partial differential equation
depends on $\vartheta $, i.e., $u=u\left(t,x,\vartheta \right)$.  The
``natural'' approximations $(\hat Z_t,\hat \Sigma _t,{\cal F}_t, 0\leq t\leq
T)$ could be constructed as follows.  Suppose that $u\left(t,x,\vartheta\right)$
is a solution of the partial differential equation
\begin{equation}
\label{1-7}
  \frac{\partial u}{\partial t}+S\left(\vartheta ,t,x\right)\frac{\partial u}{\partial
    x}+\frac{\sigma \left(\vartheta ,t,x\right)^2}{2} \frac{\partial^2 u}{\partial
    x^2}=-F\left(t,x,u, \sigma \left(t,x\right)\frac{\partial u}{\partial
    x}\right) ,
\end{equation}
with the terminal condition  $u\left(T,x,\vartheta \right)=\Phi \left(x\right)  $. 

Of course, we can not set  $Z_t=u\left(t,X_t,\vartheta \right) $ since
$\vartheta $ is unknown. One way to obtain an approximation $(\hat Z_t,\hat
\Sigma _t)$   of $\left(Z_t,\Sigma _t\right)$ is to 
find first an estimator-process $\vartheta_t^*,0< t\leq T $ and then to set
 $$
\hat Z_t=u(t,X_t,\vartheta_t^* ),\qquad \quad \hat
\Sigma _t=u'_x(t,X_t,\vartheta_t^*)\,\sigma  (\vartheta_t^*,t,X_t ).  
$$ 
If the estimator has good properties, say, $\vartheta_t^*-\vartheta $ is small
in some sense, then the error 
 \begin{align*}
\hat Z _t-Z_t=u(t,X_t,\vartheta_t^* )-u(t,X_t,\vartheta)\approx \frac{\partial
  u(t,X_t,\vartheta)}{\partial \vartheta }
\left(\vartheta_t^*-\vartheta\right)
\end{align*}
is small as well.

Here $\vartheta_t^*,0\leq t\leq T $ is a {\sl good estimator-process} of
$\vartheta $ in the sense that:
\begin{itemize}
  \item  {\it The estimator  $
    \vartheta _{t} ^*$ depends on } $X^t=\left(X_s,0\leq s\leq t\right)$.
 \item  {\it It is easy to calculate for each }$t\in (0,T]$.
 \item {\it Provides the asymptotically efficient estimation of $Z_t$, i.e.,}
  $$
\Ex_\vartheta\left(\hat Z_{t}  -Z_{t}    \right)^2\rightarrow \min .
$$ 
\end{itemize}

Therefore the main problem is {\it how to find a good estimator-process}. Such
problems were studied in the works \cite{GK15}, \cite{Kut14}, \cite{Kut16},
\cite{KZ14}.

Different problems were solved following the same
general procedure, which is illustrated as follows. Consider the forward equation
with small volatility: $\sigma \left(t,X_t,\vartheta \right)=\varepsilon \sigma
\left(t,X_t\right)$, where $\sigma \left(t,x\right)$ satisfies the conditions
\eqref{1-2}, \eqref{1-3}, $\varepsilon \in (0,1]$ is a small parameter,
  i.e. we consider asymptotics $\varepsilon \rightarrow 0$. Introduce a
  learning interval $\left[0,\tau _\varepsilon \right]$, where $\tau
  _\varepsilon \rightarrow 0$ but {\it slowly}. Using observations $X^{\tau
    _\varepsilon }=\left(X_t,0\leq t\leq \tau _\varepsilon \right)$, a
  preliminary consistent estimator $\vartheta _{\tau _\varepsilon} ^*$ of
  $\vartheta $ is constructed.  Then with the help of slightly modified
  Fisher-score device this estimator is improved up to the asymptotically
  ($\varepsilon \rightarrow 0$) efficient One-step MLE-process $\vartheta
  _{t,\varepsilon }^\star,\tau _\varepsilon ,<t\leq T$.  Now the approximation
  of $Z_t,\Sigma _t$, is given by the relations
\begin{align*}
\hat Z_t=u\left(t,X_t,\vartheta _{t,\varepsilon }^\star\right),\qquad \hat
\Sigma _t=\varepsilon \sigma \left(t,X_t\right)u'_x\left(t,X_t,\vartheta
_{t,\varepsilon }^\star\right)  
\end{align*}
It is shown that these approximations are asymptotically efficient. For the
details see \cite{Kut14}, \cite{Kut16}, \cite{KZ14}.  In \cite{GK15}
it is supposed that the volatility $\sigma \left(t,X_t,\vartheta \right)$
depends on $\vartheta $ and the forward equation  \eqref{1-1} is observed in
discrete times. Then a similar procedure of approximation was realized. 

In all the previous problems the forward equation was assumed to be
observed directly; however, in the present work we suppose that we have a partially
observed linear system, where the forward equation is {\it hidden} and we
observe its solution in the presence of  white Gaussian noise.

\section{Main result}   

\subsection{Model of observations and BSDE}
Suppose that the forward equation is 
\begin{align}
\label{2-1}
{\rm d}Y_t=-a\left(t\right)Y_t{\rm d}t+b\left(\vartheta
,t\right){\rm d}V_t,\quad  Y_0=y_0,\quad 0\leq t\leq T.
\end{align}
The solution $Y^T=\left(Y_t,0\leq t\leq T\right)$ of this equation can not
be observed directly and only the observations 
\begin{align}
\label{2-2}
{\rm d}X_t=f\left(t\right)Y_t{\rm d}t+\varepsilon \sigma
\left(t\right){\rm d}W_t,\qquad X_0=0,\quad 0\leq t\leq T
\end{align}
are available. The parameter $\vartheta \in\Theta =\left(\alpha ,\beta \right)$, where
$\left|\alpha\right|+\left|\beta \right|<\infty $.  Here $a\left(\cdot
\right), b\left( \cdot \right),$ $ f\left(\cdot \right)$ and $\sigma \left(\cdot
\right)$ are known functions and $\varepsilon \in (0,1]$. These functions
  satisfy the following regularity conditions.

{\it Conditions ${\cal A}$}.

${\cal A}_1$. {\it  The functions $a\left(t \right), b\left(\vartheta
,t\right), f\left(t \right)$ and $\sigma \left(t\right)$ have
continuous derivatives
w.r.t. $t\in\left[0,T\right] $.

${\cal A}_2$. The functions $ b\left(\vartheta
,t\right), f\left(t \right)$ and $\sigma \left(t\right)$ are
separated from zero by a constant, which does not depend neither on $\vartheta $ nor on $t$.}

We consider two functions $F(t,y,u,s)$, $\Phi \left(y \right)$ and
observations $X^T=\left(X_t,0\leq t\leq T\right)$ and we aim at constructing the
corresponding BSDE. Of course, we can not construct the
BSDE
\begin{align}
\label{2-3}
{\rm d}Z_t=-F\left(t,Y_t,Z_t,s_t\right){\rm d}t+s_t{\rm d}V_t,\quad Z_T=\Phi
\left(Y_T\right),\quad 0\leq t\leq T 
\end{align}
 for two reasons: first we have no access to the process $Y^T$ (no Wiener
 process $V_t$) and even if we
 have $Y^T$ the  solution $U=U\left(t,y,\vartheta \right)$ of the
 corresponding PDE 
\begin{align}
\label{2-4}
\frac{\partial U}{\partial t}-a\left(t\right)y\frac{\partial
  U}{\partial y}+\frac{ b\left(\vartheta
,t\right)^2 }{2} \frac{\partial^2 U}{\partial
  y^2}&=-F\left(t,y,U,b\left(\vartheta ,t\right)\frac{\partial
  U}{\partial y} \right),\nonumber\\
&\qquad U\left(T,y,\vartheta \right)=\Phi \left(y\right) 
\end{align}
 depends on the unknown parameter $\vartheta $. Therefore we can not set
 $Z_t=U\left(t,Y_t,\vartheta \right)$ since neither $Y_t$ nor
 $\vartheta $ are known.

As we have no solution $Y^T$ of the forward equation we reformulate
the problem and propose BSDE based on the best in the mean squared
estimator of this process. Introduce   the conditional expectation $\hat Y^T=(\hat
Y_t, 0\leq t\leq T)$, where  $\hat
Y_t=\Ex_{\vartheta } \left(Y_t|X_s,0\leq s\leq t\right)$. Now the
corresponding BSDE becomes
\begin{align}
\label{2-5}
{\rm d}Z_t=-F(t,\hat Y_t,Z_t,s\left(t\right)){\rm d}t+s\left(t\right){\rm
  d}\bar W_t,\qquad Z_T=\Phi (\hat Y_T),\quad 0\leq t\leq T ,
\end{align}
where the Wiener process $\bar W_t,0\leq t\leq T$ is described below. To construct
the equation \eqref{2-5} we need the equations of Kalman-Bucy filtration for
$\hat Y_t$, which we remind here. It will be convenient to denote $\hat
Y_t=m\left(\vartheta ,t\right)$ in order to show the dependence on $\vartheta $. The
equation for $m\left(\vartheta ,t\right)$ is (see \cite{LS01})
\begin{align}
\label{2-6a}
{\rm d}m\left(\vartheta ,t\right)&=-\left[a\left(t\right)
  +\frac{\gamma \left(\vartheta ,t\right)f\left(
   t\right)^2}{\varepsilon ^2\sigma \left(t\right)^2}
  \right]m\left(\vartheta ,t\right){\rm d}t
+ \frac{\gamma \left(\vartheta
  ,t\right)f\left(t\right)}{\varepsilon ^2\sigma \left(t\right)^2}
{\rm d}X_t.
\end{align}
  Here $m\left(\vartheta ,t\right)=0$  and  $\gamma
\left(\vartheta ,t\right)=\Ex_{\vartheta }\left(Y_t-m\left(\vartheta ,t\right)
\right)^2$ is the solution of Riccati equation
\begin{align}
\label{2-6b}
\frac{\partial \gamma
\left(\vartheta ,t\right)}{\partial \tau }=-2a\left(t\right)\gamma
\left(\vartheta ,t\right)-\frac{\gamma
\left(\vartheta ,t\right)^2f\left(t\right)^2}{\varepsilon ^2\sigma
  \left(t\right)^2}+b\left(\vartheta ,t\right)^2,\quad \gamma
\left(\vartheta ,0\right)=0 . 
\end{align}
We further denote
\begin{align*}
\gamma_* \left(\vartheta ,t\right)&=\frac{\gamma \left(\vartheta
  ,t\right)}{\varepsilon },\quad \gamma _0\left(\vartheta
t\right)=\frac{b\left(\vartheta ,t\right)\sigma \left(t\right)}{f\left(
 t\right)},\quad A_\varepsilon \left(\vartheta ,t\right)=\frac{\gamma_*
  \left(\vartheta ,t\right)f\left(t\right)}{\sigma
  \left(t\right)^2},\\
 A_0\left(\vartheta ,t\right)&= \frac{b\left(\vartheta
  ,t\right)}{\sigma \left(t\right)},\qquad\quad q_\varepsilon \left(\vartheta
,t\right)=a\left(t\right)+\frac{A_\varepsilon \left(\vartheta ,t\right)f\left(
 t\right)}{\varepsilon }.
\end{align*}

The true value is denoted by $\vartheta _0$. The equation \eqref{2-6a} for $m\left(\vartheta
_0,t\right)$  and   Riccati equation \eqref{2-6b} can be re-written as follows
\begin{align*}
{\rm d}m\left(\vartheta _0,t\right)&=-a\left(t\right)
m\left(\vartheta _0 ,t\right){\rm d}t+ A_\varepsilon \left(\vartheta _0
,t\right)\sigma \left(t\right) {\rm d}\bar W_t,\\ \
\frac{\partial \gamma_* \left(\vartheta _0
  ,t\right)}{\partial \tau }&=-2a\left(t\right)\gamma_*
\left(\vartheta _0 ,t\right)-\frac{A_\varepsilon \left(\vartheta_0
  ,t\right)^2\sigma \left(t\right)^2}{\varepsilon
}+\frac{b\left(\vartheta _0 ,t\right)^2}{\varepsilon },
\end{align*}
with initial values  $m\left(\vartheta _0
,0\right)=0 $ and $\gamma_*
\left(\vartheta _0 ,0\right)=0,$ respectively.   Here $\bar W_t, {\cal F}_t,$ $0\leq t\leq T$ is the {\it innovation} Wiener process
defined by the relation
\begin{align*}
{\rm d}X_t=f\left(t\right)m\left(\vartheta _0,t\right){\rm
  d}t+\varepsilon \sigma \left(t\right){\rm d}\bar W_t,\qquad X_0=0
\end{align*}
(see \cite{LS01}, Theorem 7.12).

\begin{lemma}
\label{L1} Let the conditions ${\cal A}$ be fulfilled. Then for
any $t_0\in (0,T]$ we have the convergence
\begin{align}
\label{2-6c}
\sup_{t_0\leq t\leq T}\left|\gamma _*\left(\vartheta ,t\right)- \gamma _0\left(\vartheta ,t\right)\right|\longrightarrow 0,\quad \sup_{t_0\leq t\leq T}\left|A_\varepsilon \left(\vartheta ,t\right)- A _0\left(\vartheta ,t\right)\right|\longrightarrow 0.
\end{align}
\end{lemma}
\begin{proof} See Lemma 2 in \cite{Kut19}.
\end{proof}
This lemma allows us to verify the following obvious result
\begin{lemma}
\label{L2}   Let the conditions ${\cal A}$ be fulfilled. Then for
any $t_0\in (0,T]$ we have the convergence
\begin{align}
\label{2-6d}
\sup_{t_0\leq t\leq T}\Ex_{\vartheta _0}\left|m\left(\vartheta _0,t\right)-Y_t\right|^2\leq C\varepsilon \rightarrow 0
\end{align}
as $\varepsilon \rightarrow 0$.
\end{lemma}
\begin{proof} For the difference $\delta _t=m\left(\vartheta _0,t\right)-Y_t$
  we have the equation
\begin{align*}
{\rm d}\delta _t&=-a\left(t\right)\delta _t{\rm
  d}t-b\left(\vartheta _0,t\right){\rm d}V_t +A_\varepsilon\left(\vartheta
_0,t\right)\sigma \left(t\right) {\rm d}\bar W_t\\
 &=-q_\varepsilon\left(\vartheta _0,t\right) \delta _t{\rm d}t-b\left(\vartheta _0,t\right){\rm
  d}V_t+A_\varepsilon\left(\vartheta
_0,t\right)\sigma \left(t\right){\rm d}W_t,
\end{align*}
where 
\begin{align*}
q_\varepsilon \left(\vartheta _0,t\right)=a\left(t\right)+{\varepsilon }^{-1}{A_\varepsilon\left(\vartheta
_0,t\right)f\left(t\right) } .
\end{align*}
Hence
\begin{align}
\label{2-6e}
\delta _t&=-\int_{0}^{t}e^{-\int_{s}^{t}q_\varepsilon \left(\vartheta
  _0,v\right){\rm d}v} b\left(\vartheta _0,s\right){\rm d}V_s\nonumber\\
&\qquad \qquad \qquad +\int_{0}^{t}e^{-\int_{s}^{t}q_\varepsilon \left(\vartheta
  _0,v\right){\rm d}v} A_\varepsilon \left(\vartheta _0,s\right)\sigma \left(s\right){\rm d}W_s
\end{align}
and
\begin{align*}
&\Ex_{\vartheta _0}\left|m\left(\vartheta
_0,t\right)-Y_t\right|^2=\int_{0}^{t}e^{-2\int_{s}^{t}q_\varepsilon
  \left(\vartheta _0,v\right){\rm d}v} \left[b\left(\vartheta
  _0,s\right)^2+A_\varepsilon \left(\vartheta _0,s\right)^2\sigma
  \left(s\right)^2\right]{\rm d}s \\
 &\qquad \leq C\int_{0}^{t}e^{-\frac{2}{\varepsilon }\int_{s}^{t}A_\varepsilon
  \left(\vartheta _0,v\right)f\left( _0,v\right){\rm d}v}{\rm d}s
\leq C\int_{0}^{t}e^{-\frac{c\left(t-s\right)}{\varepsilon }}{\rm d}s \leq
C\varepsilon \left[1-e^{-\frac{ct}{\varepsilon }}\right] .
\end{align*}
Here we used the condition ${\cal A}_2$ and the boundedness of all functions.

\end{proof}
Therefore for small $\varepsilon $ the random process $m\left(\vartheta
_0,t\right)$ is a good approximation of the solution $Y_t$ of the forward equation.

It is worth mentioning that if $\vartheta _0$ is known, then in order to construct  \eqref{2-5} with
innovation Wiener process we need the solution of the partial differential
equation
\begin{align}
\label{2-7}
\frac{\partial u}{\partial t}-a\left(t\right)y\frac{\partial
  u}{\partial y}+\frac{ B_\varepsilon \left(\vartheta _0 ,t\right)^2 }{2} \frac{\partial^2 u}{\partial
  y^2}&=-F\left(t,y,u,B_\varepsilon \left(\vartheta _0 ,t\right)  \frac{\partial
  u}{\partial y}  \right),\nonumber\\
 u\left(T,y,\vartheta_0,\varepsilon  \right)&=\Phi\left(y\right),
\end{align}
where $B_\varepsilon \left(\vartheta _0 ,t\right)=A_\varepsilon
\left(\vartheta _0 ,t\right)\sigma \left(t\right)$. If the solution
of this equation was available, then
$$
Z_t=u\left(t,m\left(\vartheta _0,t\right),\vartheta _0,\varepsilon\right),\;
s\left(t\right)=  A_\varepsilon \left(\vartheta _0 ,t\right)  \sigma
  \left(t\right) \frac{\partial u}{\partial y}\left(t,m\left(\vartheta
  _0,t\right),\vartheta _0,\varepsilon\right) 
$$
would form the equation \eqref{2-5}.

We further denote
\begin{align*}
\frac{\partial }{\partial t}u\left(t,y,\vartheta
,\varepsilon\right)&=u_t'\left(t,y,\vartheta ,\varepsilon\right),\qquad
\frac{\partial }{\partial y}u\left(t,y,\vartheta
,\varepsilon\right)=u_y'\left(t,y,\vartheta ,\varepsilon\right),\\
\frac{\partial }{\partial \vartheta }u\left(t,y,\vartheta
,\varepsilon\right)&=\dot u\left(t,y,\vartheta ,\varepsilon\right),\qquad
\frac{\partial }{\partial \varepsilon }u\left(t,y,\vartheta
,\varepsilon\right)=u_\varepsilon '\left(t,y,\vartheta ,\varepsilon\right).
\end{align*}

Suppose that we have some estimator-process $\vartheta _{t,\varepsilon }^*$
which is consistent :  for any $t\in (0,T]$ the estimator $\vartheta
  _{t,\varepsilon }^*\rightarrow \vartheta _0$. Let us set
$$
\hat Z_t=u\left(t,\hat m_t,\vartheta_{t,\varepsilon }^*,\varepsilon\right),\qquad 
\hat s\left(t\right)=  A_\varepsilon \left(\vartheta_{t,\varepsilon }^* ,t\right)  \sigma
  \left(t\right) u'\left(t,\hat m_t,\vartheta_{t,\varepsilon }^*,\varepsilon\right),
$$ 
where $\hat m_t$ is an approximation of $m\left(\vartheta _0,t\right)$. 
Then we have the following question: what is the relation between the solution
  $Z_t$ of the equations \eqref{2-3}, the solution
  $Z_t$ of the equations \eqref{2-5} and the approximation $\hat Z_t $
introduced above? 

The convergences
\begin{align*}
\gamma _*\left(\vartheta,t \right)\longrightarrow \gamma _0\left(\vartheta,t
\right),\qquad A_\varepsilon \left(\vartheta,t \right)\longrightarrow
A_0\left(\vartheta,t \right),\qquad m\left(\vartheta
_0,t\right)\longrightarrow Y_t,
\end{align*}
imply that the coefficient $B_\varepsilon \left(\vartheta,t \right)^2 $ in the equation
 \eqref{2-7} converges to $b\left(\vartheta ,t\right)^2$ in the equation
 \eqref{2-4}. Hence under regularity conditions the solution $u\left(\cdot
 ,\cdot ,\cdot \right)$ of \eqref{2-7} converges to the solution $u\left(\cdot
 ,\cdot ,\cdot \right)$ of \eqref{2-4}.

Note that we have no BSDE for the approximation process
\begin{align*}
{\rm d}\hat Z_t=-F(t,\hat Z_t,\hat m_t, \hat
s\left(t\right) ){\rm d}t+ \hat s\left(t\right){\rm d}\bar W_t,\qquad \hat Z_T=\Phi
(\hat m_T ).
\end{align*}
The stochastic differential for the random process $\hat
Z_t=u\left(t,\hat m_t,\vartheta_{t,\varepsilon }^*,\varepsilon \right) $ could be written analytically (it is
different of given above), but it is
quite cumbersome and it is not used in the proofs. Our goal is to propose an
approximation of the solution $Z_t$ of the equation \eqref{2-3} and to study
the error of approximation, say, $\Ex_{\vartheta _0}(Z_t-\hat Z_t
)^2$. Moreover, the optimality of such approximation is discussed.

\subsection{Preliminary estimators}

Our objective is to construct a {\it good estimator-process} $\left(\vartheta
_{t,\varepsilon },\tau \leq t\leq T\right)$ and for this construction we need
a preliminary estimator $\bar\vartheta _{\tau ,\varepsilon }$ constructed by
the first observations $X^{\tau }=\left(X_t,0\leq t\leq \tau 
\right)$ on the (small) time interval $\left[0,\tau \right] $ where  $\tau \in
( 0,T] $.  In this section we propose two such estimators. One is the MLE
  $\hat\vartheta _{\tau ,\varepsilon }$ and the other is the estimator of
  substitution which uses the estimator of the quadratic variation of the
  derivative of the limit of the observed process. 

The likelihood ratio function is (see \cite{LS01})
\begin{align*}
L\left(\vartheta ,X^\tau \right)=\exp\left\{\int_{0}^{\tau
  }\frac{f\left(t\right)m\left(\vartheta ,t\right)}{\varepsilon
    ^2\sigma \left(t\right)^2}{\rm d}X_t-\int_{0}^{\tau
  }\frac{f\left(t\right)^2m\left(\vartheta
    ,t\right)^2}{2\varepsilon ^2\sigma \left(t\right)^2}{\rm
  d}t\right\},\vartheta \in\Theta , 
\end{align*}
and  the corresponding  maximum likelihood estimator (MLE) $\hat\vartheta _{\tau,\varepsilon  }$
 is defined  by 
\begin{align}
\label{2-8}
L(\hat\vartheta _{\tau,\varepsilon  }  ,X^\tau )=\sup_{\vartheta \in\Theta
}L\left(\vartheta ,X^\tau \right) .
\end{align}

In the sequel, let us introduce the notation
\begin{align*}
 {\rm I}^\tau \left(\vartheta \right)&=\int_{0}^{\tau
}\frac{f\left(t\right)\dot b\left(\vartheta ,t\right)^2}{2b\left(\vartheta ,t\right)\sigma
  \left(t\right)}\;{\rm d}t ,\quad 
G_\tau\left(\vartheta ,\vartheta _0\right)=\int_{0}^{\tau
}\frac{f\left(t\right)\left[ b\left(\vartheta ,t\right)- b\left(\vartheta_0
    ,t\right)\right]^2}{2b\left(\vartheta ,t\right)\sigma 
  \left(t\right)}\;{\rm d}t.
\end{align*}

{\it Conditions ${\cal B}$}.

${\cal B}_1$. {\it  The function $ b\left(\vartheta
,t\right)$ has three continuous derivatives
w.r.t. $\vartheta\in \Theta  $.}

${\cal B}_2$. {\it Identifiability condition: For any $\tau \in (0,T]$ and $\nu >0$
\begin{align*}
\inf_{\vartheta _0\in\Theta }\inf_{\left|\vartheta -\vartheta _0\right|>\nu
}G_\tau\left(\vartheta ,\vartheta _0\right) >0.
\end{align*}

${\cal B}_3$. Non degeneracy of Fisher information: For any $\tau \in (0,T]$
\begin{align*}
\inf_{\vartheta\in\Theta }{\rm I}^\tau \left(\vartheta \right)>0.
\end{align*}

}

Note that if $f\left(0\right)>0$ and $\inf_{\vartheta \in\Theta } \left|\dot
b\left(\vartheta ,0\right)b\left(\vartheta ,0\right)^{-1}\right|>0$, then the
condition ${\cal B}_3$ is fulfilled.

\begin{proposition}
\label{P1}
The MLE $\hat\vartheta _{\tau ,\varepsilon }$  under regularity conditions
${\cal A}$, ${\cal B}$ is consistent, asymptotically 
normal
\begin{align}
\label{2-9}
\sqrt{\frac{{\rm I}^\tau \left(\vartheta_0 \right)}{\varepsilon
}}  \left(\hat\vartheta _{\tau,\varepsilon  }-\vartheta
_0\right)\Longrightarrow \zeta \sim {\cal N}\left(0,1 \right), 
\end{align}
asymptotically efficient and the moments converge: for any $p>0$ 
\begin{align}
\label{2-10}
 \left|\frac{{\rm I}^\tau \left(\vartheta_0 \right)}{\varepsilon
}\right|^{p/2}  \Ex_{\vartheta _0} \left|\hat\vartheta _{\tau,\varepsilon  }-\vartheta
_0\right|^p \longrightarrow  \Ex\left|\zeta \right|^p.
\end{align}
\end{proposition}
For the proof see  \cite{Kut19}, Theorem 1. 

Note that we can not use the MLE-process $\hat\vartheta _{t,\varepsilon } ,
0<t\leq T$ as a good estimator-process since in order to solve equation \eqref{2-8} for
all $t\in (0,T]$ we need the solutions $m\left(\vartheta ,s\right), 0\leq
  s\leq t$ of equations \eqref{2-6a} for all $\vartheta \in\Theta $ and all
  $t\in (0,T]$. From a computational point of view, a good estimator-process would be one that
  could be easily computed. Therefore, $ \hat\vartheta _{\tau,\varepsilon  }$ could be considered as a
    preliminary estimator. Its calculation is simpler because we need to solve
    \eqref{2-8} just once. 

Note that even the calculation of the preliminary MLE $\hat\vartheta _{\tau
  ,\varepsilon }$ by  \eqref{2-8} requires the solution of the filtration
equations for many values of $\vartheta $. Below we propose another estimator
which requires much more simple  calculations.

Let us consider another estimation procedure of the preliminary estimator
based on the following property of the model. Remark that the observed process $X^\tau $
converges with probability 1 to $ x^\tau $:
\begin{align*}
\sup_{0\leq t\leq \tau }\left|X_t- x_t\right|\longrightarrow 0,
\end{align*}
where $x^\tau=\left(x_t,0\leq t\leq \tau \right) $ satisfies the limit ($\varepsilon =0$) relation
\begin{align*}
x_t=\int_{0}^{t} f\left(s\right)Y_s\,{\rm d}s,\qquad \qquad 0\leq t\leq \tau .
\end{align*} 
Here $Y_t$ is solution of the forward equation \eqref{2-1}. Let us put  $
N_t=f\left(t\right)Y_t$. Then by It\^o formula
\begin{align*}
N_\tau ^2-2\int_{0}^{\tau } N_t{\rm d}N_t=\int_{0}^{\tau }f\left(
t\right)^2b\left(\vartheta _0,t\right)^2{\rm d}t.
\end{align*}
We further define
\begin{align*}
\Psi \left(\vartheta \right)&=\int_{0}^{\tau } f\left(t\right)^2
b\left(\vartheta ,t\right)^2{\rm d}t,\qquad \bar N_{\tau ,\varepsilon
}=\frac{1}{\varphi _\varepsilon }\int_{0}^{\tau }K_*\left(\frac{s-t }{\varphi
  _\varepsilon }\right){\rm d}X_s ,\\ N_{t ,\varepsilon }&=\frac{1}{\varphi
  _\varepsilon }\int_{0}^{\tau }K\left(\frac{s-t }{\varphi _\varepsilon
}\right){\rm d}X_s ,\quad 0\leq t\leq \tau ,\\ \hat\Psi _\varepsilon& =\bar N_{\tau
  ,\varepsilon } ^2-2\int_{0}^{\tau } N_{t ,\varepsilon }{\rm d}N_{t
  ,\varepsilon },\qquad \dot \Psi \left(\vartheta \right)=2\int_{0}^{\tau } f\left(t\right)^2
b\left(\vartheta ,t\right)\dot b\left(\vartheta ,t\right){\rm d}t.
\end{align*}
Here the one-sided  kernels $K_*\left(\cdot \right)$ and $K\left(\cdot
\right)$ satisfy the usual conditions 
\begin{align*}
&K_*\left(u\right)\geq 0,\qquad \int_{-1}^{0} K_*\left(u\right){\rm d}u=1,\qquad
K_*\left(u\right)=0,\;{\rm for}\; u\not\in \left[-1,0\right],\\
&K\left(u\right)\geq 0,\qquad \int_{0}^{1} K\left(u\right){\rm d}u=1,\quad \qquad
K\left(u\right)=0,\;{\rm for}\; u\not\in \left[0,1\right] .
\end{align*}
Further, suppose that the function  $\Psi \left(\vartheta \right),\vartheta
\in\Theta  $ is monotone increasing and denote 
\begin{align*}
\psi _m&=\inf _{\vartheta \in\Theta }\Psi \left(\vartheta \right),\quad \psi
_M=\sup _{\vartheta \in\Theta }\Psi \left(\vartheta \right) ,\qquad \psi
_m=\Psi \left(\alpha  \right),\quad \psi_M=\Psi \left(\beta   \right), \\
G\left(\psi \right)&=\Psi ^{-1}\left(\psi \right),\quad \psi _m<\psi <\psi
_M,\qquad \alpha <G\left(\psi \right)<\beta ,\qquad \eta _\varepsilon
=G(\hat\Psi _\varepsilon ), \\ 
\BB_m&=\left\{\omega :\quad \hat\Psi_\varepsilon \leq  \psi _m \right\},\qquad
\BB_M=\left\{\omega :\quad \hat\Psi_\varepsilon \geq \psi _M \right\}, \\
\BB&=\left\{\omega :\quad \psi _m <\hat\Psi_\varepsilon < \psi _M
\right\},\qquad g\left(\nu \right)=\inf_{\vartheta _0\in\Theta }\inf_{\left|\vartheta
  -\vartheta _0\right|>\nu }\left|\Psi \left(\vartheta 
\right)-\Psi \left(\vartheta_0 \right)\right|. 
\end{align*}
The    substitution estimator (SE) is introduced as follows
\begin{align}
\label{2-10a}
\check \vartheta _{\tau ,\varepsilon }&=\alpha \1_{\left\{\BB_m\right\}}+\eta
_\varepsilon\1_{\left\{\BB\right\}}+\beta \1_{\left\{\BB_M\right\}} .
\end{align}
It has the following properties. 

\begin{proposition}
\label{P2} Suppose that the conditions  ${\cal A}_1,{\cal B}_1 $ are  fulfilled, for any
(small) $\nu >0$ we have  $g\left(\nu \right)>0$ and $\inf_{\vartheta
  \in\Theta }\dot \Psi \left(\vartheta \right)>0$. Then the SE $\check
\vartheta _{\tau ,\varepsilon } $ is uniformly consistent and for any $p>0$
there exists a constant $C=C\left(p\right)>0$ such that
\begin{align}
\label{2-10b}
\sup_{\vartheta _0\in\Theta }\varepsilon ^{-p/2}\Ex_{\vartheta _0}\left|\check
\vartheta _{\tau ,\varepsilon }-\vartheta _0\right|^p\leq C .
\end{align}
\end{proposition}
For the proof and more general results see \cite{Kut20a}.

\subsection{One-step MLE-process}

Below we consider the MLE $\hat\vartheta _{\tau ,\varepsilon }$ as a preliminary
estimator. Following the same steps it could be shown that the SE
$\check\vartheta _{\tau ,\varepsilon }$ could also be used as the preliminary. Recall that this estimator is easier to calculate and the property
\eqref{2-10b} is sufficient for the proof of the Proposition \ref{P3}. 

Let us introduce the statistic
\begin{align*}
\vartheta _{t,\varepsilon }^\star=\hat\vartheta _{\tau,\varepsilon
}+\frac{1}{{\rm I}_\tau^t (\hat\vartheta _{\tau,\varepsilon  } )} \int_{\tau
}^{t}\frac{f(s)\dot m(\hat\vartheta _{\tau,\varepsilon  },s)}{{\varepsilon}
  \sigma
    \left(s\right)^2} \left[{\rm d}X_s-f(s)m(\hat\vartheta _{\tau,\varepsilon
  },s ){\rm d}s\right],
\end{align*}
where  ${\rm I}_\tau^t \left(\vartheta \right)$ is  the Fisher information  
\begin{align*}
{\rm I}_\tau^t \left(\vartheta \right)=\int_{\tau}^{t
}\frac{f\left(s\right)\dot b\left(\vartheta ,s\right)^2}{2b\left(\vartheta ,s\right)\sigma
  \left(s\right)}\;{\rm d}s.
\end{align*}
We have now to precise how to calculate the values $\dot m(\hat\vartheta
_{\tau,\varepsilon  },s) $ and $  m(\hat\vartheta _{\tau,\varepsilon  },s)$
since according to \eqref{2-6a} we have 
\begin{align*}
 m(\vartheta ,t)&=\int_{0}^{t}e^{-\int_{s}^{t}q_\varepsilon \left(\vartheta ,v\right){\rm
     d}v} \frac{\gamma \left(\vartheta ,s\right)f\left(s\right)}{\varepsilon
   ^2\sigma \left(s\right)^2}{\rm d}X_s\\
&=e^{-\int_{0}^{t}q_\varepsilon \left(\vartheta ,v\right){\rm
     d}v}\int_{0}^{t}e^{\int_{0}^{s}q_\varepsilon \left(\vartheta ,v\right){\rm
     d}v} \frac{\gamma \left(\vartheta ,s\right)f\left(s\right)}{\varepsilon
   ^2\sigma \left(s\right)^2}{\rm d}X_s \\
&=h\left(\vartheta ,t\right)\int_{0}^{t}H\left(\vartheta ,s\right){\rm d}X_s.
\end{align*}
As usual in such situations we replace the stochastic
integral by an ordinary one as follows. We have
\begin{align*}
\int_{0}^{t}H\left(\vartheta ,s\right){\rm d}X_s =H\left(\vartheta ,t\right)
X_t-\int_{0}^{t}X_s H'_s\left(\vartheta ,s\right){\rm d}s.
\end{align*}
Let us denote by $N\left(\vartheta ,t,X^t\right)$ the right hand side of this
equality. Then we can set
\begin{align*}
m(\hat\vartheta _{\tau,\varepsilon
  } ,s)=h(\hat\vartheta _{\tau,\varepsilon
},s)N(\hat\vartheta _{\tau,\varepsilon  } ,s,X^s). 
\end{align*}
The similar relation could also be written for $\dot m(\hat\vartheta _{\tau,\varepsilon
  } ,s) $. 

Further, introduce the  random processes
\begin{align*}
&\eta _{t,\varepsilon }=\frac{\vartheta _{t,\varepsilon }^\star -\vartheta
  _0}{\sqrt{\varepsilon }},\qquad \qquad   \tau \leq t\leq T,\\
&\eta _t=\frac{1}{{\rm I}_\tau^t (\vartheta _0 )} \int_{\tau
}^{t}\frac{\dot b\left(\vartheta _0,s\right) \sqrt{f\left(s\right)}
  }{\sqrt{2b\left(\vartheta _0,s\right)\sigma \left(s\right)}}\;{\rm
    d}w\left(s\right), \qquad   \tau \leq t\leq T, 
\end{align*}
where $w\left(s\right),0 \leq s\leq T$ is some standard Wiener process.

In the sequel, we need an additional condition.

${\cal B}_4$.{\it  Non degeneracy of Fisher information: For any $t_0 \in (\tau,T]$}
\begin{align*}
\inf_{\vartheta\in\Theta }{\rm I}^{t_0}_\tau \left(\vartheta \right)>0.
\end{align*}

\begin{proposition}
\label{P3} Let the conditions ${\cal A}$, ${\cal B}$ be fulfilled. Then the
One-step MLE-process $\vartheta _{t,\varepsilon }^\star,\tau <t\leq T$ is
uniformly consistent: for any $\nu >0$ and any $t_0\in (\tau ,T)$
\begin{align}
\label{2-11}
\Pb_{\vartheta _0}\left(\sup_{t_0\leq t\leq T}\left|\vartheta _{t,\varepsilon
}^\star-\vartheta _0 \right|>\nu \right)\longrightarrow 0, 
\end{align}
the stochastic process $\eta _{t,\varepsilon }, t_0\leq t\leq T$ converges in
distribution in the measurable space $\left({\cal C}\left[t_0,T\right],{\scr
  B}\right)$ to the random process $\eta _{t }, t_0\leq t\leq T$
\begin{align}
\label{2-12}
\eta _{\cdot ,\varepsilon }\Longrightarrow \eta _{\cdot  },\qquad \eta _{t
}\sim {\cal N}\left(0, {\rm I}_\tau^t (\vartheta _0 )^{-1}\right).
\end{align}
\end{proposition}
\begin{proof}
Consider the normalized difference
\begin{align*}
\frac{\vartheta _{t,\varepsilon }^\star -\vartheta _0}{\sqrt{\varepsilon }}
&=\frac{\hat\vartheta _{\tau,\varepsilon }-\vartheta _0}{ \sqrt{\varepsilon
}}+\frac{1}{{\rm I}_\tau^t (\hat\vartheta _{\tau,\varepsilon } )} \int_{\tau
}^{t}\frac{f\left(s\right)\dot m(\hat\vartheta _{\tau,\varepsilon },s)}{
  \sqrt{\varepsilon }\sigma \left(s\right)} {\rm d}\bar W_s\\ 
&\qquad +\frac{1}{{\rm I}_\tau^t (\hat\vartheta _{\tau,\varepsilon } )} \int_{\tau
}^{t}\frac{f\left(s\right)^2\dot m(\hat\vartheta _{\tau,\varepsilon },s)}{{\varepsilon^{3/2}}
  \sigma     \left(s\right)^2}\left[m\left(\vartheta
  _0,s\right)- m(\hat\vartheta _{\tau,\varepsilon  },s )   \right] {\rm d}s.
\end{align*}

 We have the relations (Lemma 6 in \cite{Kut19})
\begin{align*}
&m(\hat\vartheta _{\tau,\varepsilon },s )-m\left(\vartheta _0,s\right)=
(\hat\vartheta _{\tau,\varepsilon }-\vartheta _0)\dot
m(\tilde\vartheta _{\tau,\varepsilon },s ),\\
&m(\hat\vartheta _{\tau,\varepsilon },s )-m\left(\vartheta _0,s\right)=
(\hat\vartheta _{\tau,\varepsilon }-\vartheta _0)\dot
m(\vartheta _0,s )+\frac{1}{2}(\hat\vartheta
_{\tau,\varepsilon }-\vartheta _0)^2\ddot
m(\tilde\vartheta _{\tau,\varepsilon },s ),\\
&\dot m(\vartheta _0,s )=\sqrt{{\frac{\varepsilon \sigma
    \left(s\right)}{2b\left(\vartheta _0,s\right)  f\left(s\right)  }}}\;  \dot
    b\left(\vartheta _0,s\right)
  \xi _{s,\varepsilon }+{\varepsilon
}\; R_{t,\varepsilon },\\
&{\rm I}_\tau^t (\hat\vartheta _{\tau,\varepsilon }) ^{-1}={\rm
    I}_\tau^t (\vartheta _0 ) ^{-1}+(\hat\vartheta
_{\tau,\varepsilon }-\vartheta _0)\;Q_{t,\varepsilon }.
\end{align*}
Here $ \xi _{s,\varepsilon },s \in \left[\tau ,T\right]$ are Gaussian, asymptotically
independent random variables, i.e.,  $\xi _{s,\varepsilon }\Longrightarrow \xi _s\sim
{\cal N}\left(0,1\right)$, where $\xi _{s },s\in \left[\tau ,T\right]$ are
mutually independent. 

For any $\nu >0$ we can write 
\begin{align*}
&\Pb_{\vartheta _0}\left(\sup_{t_0\leq t\leq T}\left|\vartheta _{t,\varepsilon
}^\star-\vartheta _0 \right|>\nu \right)\leq \Pb_{\vartheta _0}\left(
\left|\hat\vartheta _{\tau,\varepsilon }-\vartheta _0\right|\geq \frac{\nu
}{3}\right)\\
&\qquad + \Pb_{\vartheta _0}\left(
{\rm I}_\tau^{t_0} (\hat\vartheta _{\tau,\varepsilon })^{-1}\left|\int_{\tau
}^{T}\frac{f\left(s\right)\dot m(\hat\vartheta _{\tau,\varepsilon },s)}{
  \sigma \left(s\right)} {\rm d}\bar W_s\right|\geq \frac{\nu
}{3}\right)\\
&\qquad + \Pb_{\vartheta _0}\left( \frac{
\left|\hat\vartheta _{\tau,\varepsilon }-\vartheta _0\right|  }{ {\rm
    I}_\tau^{t_0} (\hat\vartheta _{\tau,\varepsilon })} \int_{\tau
}^{T}\frac{f\left(s\right)^2\left|\dot m(\hat\vartheta _{\tau,\varepsilon },s)\dot
  m(\tilde\vartheta _{\tau,\varepsilon },s)\right|}{ \sqrt{\varepsilon }
  \sigma \left(s\right)^2} {\rm d}s\geq \frac{\nu
}{3}\right). 
\end{align*}
Now the convergence \eqref{2-11} follows from the consistency of
$\hat\vartheta _{\tau,\varepsilon }  $ 
and the following estimate of the moments of $\dot m\left(\cdot ,\cdot  
\right)$: for any $p>0$ 
\begin{align}
\label{2-13}
\sup_{\vartheta _0\in \Theta }\sup_{\tau \leq t\leq T}\Ex_{\vartheta
  _0}\left|\dot m\left(\vartheta _0 ,t\right)\right|^p\leq  C\varepsilon ^{p/2}.
\end{align}
The proof of this estimate follows from the proof of Lemma 6 in \cite{Kut19}. 

Moreover, we have the convergence results
\begin{align*}
\int_{\tau}^{t}\frac{f\left(s\right)^2\dot m(\hat\vartheta _{\tau,\varepsilon
  },s)^2}{ {\varepsilon 
    }\sigma \left(s\right)^2} {\rm d}s&=\int_{\tau}^{t 
}\frac{f\left(s\right)\dot b\left(\vartheta _0  ,s\right)^2\xi _{s,\varepsilon
  }^2}{2b\left(\vartheta _0  ,s\right)\sigma 
  \left(s\right)}\;{\rm d}s\left(1+o\left(1\right)\right)\longrightarrow {\rm
    I}_\tau^t (\vartheta _0 ),\\
 \int_{\tau
}^{t}\frac{f\left(s\right)\dot m(\hat\vartheta _{\tau,\varepsilon },s)}{ \sqrt{\varepsilon
    }\sigma \left(s\right)} {\rm d}\bar W_s&= \int_{\tau
}^{t}\frac{\sqrt{f\left(s\right)}\;\dot b(\vartheta _0,s)\,\xi _{s,\varepsilon
 }}{\sqrt{2b(\vartheta _0,s)  
    \sigma \left(s\right)}} {\rm d}\bar W_s\left(1+o\left(1\right)\right)\\
&\qquad \qquad\qquad \qquad\qquad \qquad \Longrightarrow {\cal N}\left(0,{\rm
    I}_\tau^t (\vartheta _0 ) \right).
\end{align*}
The random processes $R_{t,\varepsilon },Q_{t,\varepsilon }$ have bounded
polynomial moments. 

Hence we can write the representation
\begin{align}
\label{2-14}
&\frac{\vartheta _{t,\varepsilon }^\star -\vartheta _0}{\sqrt{\varepsilon }}
=\frac{\hat\vartheta _{\tau,\varepsilon }-\vartheta _0}{ \sqrt{\varepsilon
}}+\frac{1}{{\rm I}_\tau^t (\vartheta _0 )} \int_{\tau }^{t}\frac{\dot
  b(\vartheta _0,s)\sqrt{f\left(s\right)}\,\xi _{s,\varepsilon
}}{\sqrt{2b(\vartheta _0,s) \sigma 
  \left(s\right)}} {\rm d}\bar W_s+o\left(1\right)\nonumber\\ 
&\qquad \quad \qquad -\frac{(\hat\vartheta
  _{\tau,\varepsilon }-\vartheta _0)}{ \sqrt{\varepsilon }} \;\frac{1}{{\rm
    I}_\tau^t (\vartheta _0 )} \int_{\tau}^{t 
}\frac{\dot b\left(\vartheta _0 ,s\right)^2f\left(s\right)\,\xi _{s,\varepsilon
  }^2}{2b\left(\vartheta _0 ,s\right)\sigma \left(s\right)}\;{\rm
  d}s+\frac{(\hat\vartheta 
  _{\tau,\varepsilon }-\vartheta _0)^2}{ \sqrt{\varepsilon }}P_{t,\varepsilon }\nonumber\\ 
&\quad =\frac{1}{{\rm I}_\tau^t (\vartheta _0 )} \int_{\tau }^{t}\frac{\dot
  b(\vartheta _0,s)\sqrt{f\left(s\right)}\xi _{s,\varepsilon
}}{\sqrt{2b(\vartheta _0,s) \sigma 
  \left(s\right)}}\; {\rm d}\bar W_s+o\left(1\right)+\left(\frac{\hat\vartheta
  _{\tau,\varepsilon }-\vartheta _0}{ \sqrt{\varepsilon
}}\right)^2P_{t,\varepsilon }\sqrt{\varepsilon } .
\end{align}
The random processes $P_{t,\varepsilon }$ has bounded polynomial moments. From
this representation it follows that the One-step MLE-process is asymptotically
normal: for all $t\in (\tau ,T]$
\begin{align*}
\frac{\vartheta _{t,\varepsilon }^\star -\vartheta _0}{\sqrt{\varepsilon
}}\Longrightarrow \eta _t \sim {\cal N}\left(0,{\rm I}_\tau^t (\vartheta _0 )^{-1} \right).
\end{align*}
The representation \eqref{2-14} allows us to obtain the convergence of the finite
dimensional distributions as well
\begin{align}
\label{2-15}
\left(\eta _{t_1,\varepsilon },\ldots,\eta _{t_k,\varepsilon
}\right)\Longrightarrow \left(\eta _{t_1 },\ldots,\eta _{t_k  }\right),
\end{align}
for any $k\geq 2$ and any $t_0\leq t_1<\ldots <t_k\leq T$. 

Let us verify the condition
\begin{align}
\label{2-16}
\Ex_{\vartheta _0}\left|\eta _{t_1,\varepsilon }-\eta _{t_2,\varepsilon
}\right|^4\leq C\left|t_1-t_2\right|^2 
\end{align} 
which along with the convergence  \eqref{2-15} provide the weak
convergence  \eqref{2-12} of
the random process $\eta _{t,\varepsilon },t_0\leq t\leq T$. We introduce
\begin{align*}
J_1\left(t\right)&=\int_{\tau }^{t}\frac{f\left(s\right)\dot m(\hat\vartheta
  _{\tau ,\varepsilon },s)}{\sqrt{\varepsilon }\sigma \left(s\right)}
{\rm d}\bar W_s,\qquad K\left(t\right)={\rm I}_\tau^t (\hat\vartheta
  _{\tau ,\varepsilon } )^{-1},\\
J_2\left(t\right)&=\int_{\tau }^{t}\frac{f\left(s\right)^2\dot m(\vartheta
  _{\tau ,\varepsilon },s)[m(\vartheta  _{0 },s)-m(\hat\vartheta 
  _{\tau ,\varepsilon },s) ]}{{\varepsilon }\sigma \left(s\right)^2}{\rm d}s
\end{align*}
Then we can write
\begin{align*}
\Ex_{\vartheta _0}\left|\eta _{t_1,\varepsilon }-\eta _{t_2,\varepsilon
}\right|^4
&\leq C\Ex_{\vartheta _0}\left|K\left(t_1\right)
J_1\left(t_1\right)-K\left(t_2\right) J_1\left(t_2\right)
\right|^4\\
&\qquad +C\Ex_{\vartheta _0}\left|K\left(t_1\right)
J_2\left(t_1\right)-K\left(t_2\right) J_2\left(t_2\right)
\right|^4\\
&\leq C\Ex_{\vartheta _0}\left|\left(K\left(t_1\right)
-K\left(t_2\right)\right) J_1\left(t_1\right)
\right|^4\\
&\qquad +  C\Ex_{\vartheta _0}\left|
\left(J_1\left(t_1\right)- J_1\left(t_2\right)\right) K\left(t_2\right)
\right|^4\\
&\qquad +C\Ex_{\vartheta _0}\left|\left(K\left(t_1\right)
-K\left(t_2\right)\right) J_2\left(t_1\right)
\right|^4\\
&\qquad +  C\Ex_{\vartheta _0}\left|
\left(J_2\left(t_1\right)- J_2\left(t_2\right)\right) K\left(t_2\right)
\right|^4.
\end{align*}
Using once again the estimates \eqref{2-13} we obtain
\begin{align*}
\Ex_{\vartheta _0}\left|K\left(t_1\right) -K\left(t_2\right) \right|^8&\leq
C\left|t_2-t_1\right|^8,\qquad  \Ex_{\vartheta _0}\left|J_1\left(t_2\right)
\right|^8\leq C,\\
\Ex_{\vartheta _0}\left|
\left(J_1\left(t_1\right)- J_1\left(t_2\right)\right) 
\right|^8&\leq  C\left|t_2-t_1\right|^4,\qquad  \Ex_{\vartheta _0}\left|K\left(t_2\right)
\right|^8\leq C,\\ \Ex_{\vartheta _0}\left|
\left(J_2\left(t_1\right)- J_2\left(t_2\right)\right)
\right|^8&\leq C\left|t_2-t_1\right|^8.
\end{align*}
These estimates and the Cauchy-Schwartz inequality allow us to verify \eqref{2-16} and
therefore to obtain \eqref{2-12}.

\end{proof}

\subsection{Approximation}

Consider the family of solutions $u\left(t,y,\vartheta,\varepsilon \right),\vartheta
\in\Theta, \varepsilon \in (0,1] $ of the equations
\begin{align}
\label{3-1}
u'_t-a\left(t\right)y\,u'_y+\frac{ 1 }{2}B_\varepsilon \left(\vartheta
,t\right)^2 u''_{yy}&=-F\left(t,y,u,B_\varepsilon \left(\vartheta
,t\right)u'_y
\right),\nonumber\\ u\left(T,y,\vartheta,\varepsilon\right)&=\Phi\left(y\right).
\end{align}
and the equation
\begin{align}
\label{3-2}
U'_t-a\left(t\right)y\,U'_y+\frac{ 1 }{2}b \left(\vartheta
,t\right)^2 U''_{yy}&=-F\left(t,y,U,b \left(\vartheta
,t\right)U'_y
\right),\nonumber\\ U\left(T,y,\vartheta\right)&=\Phi\left(y\right).
\end{align}

We suppose that by continuity $U\left(t,y,\vartheta
\right)=u\left(t,y,\vartheta,0 \right)$. Recall that $B_\varepsilon
\left(\vartheta ,t\right)\rightarrow b \left(\vartheta ,t\right) $ as
$\varepsilon \rightarrow 0$. 

{\it  Conditions ${\cal C}$}

${\cal C}_1$. {\it The functions $F\left(t,y,u,s\right) $ and $\Phi
  \left(y\right)$ satisfy the conditions \eqref{1-5}, \eqref{1-6}. }

${\cal C}_2$. {\it The function  $u \left(t,y,\vartheta,\varepsilon
  \right),t\in(0,T],y\in {\cal R},\vartheta \in\Theta ,\varepsilon \in
\left[0,1\right]$ has continuous derivatives $u'_y\left(\cdot \right),\dot u
\left(\cdot \right),u'_\varepsilon \left(\cdot \right) $. }

\bigskip

It is worth noting that $Z_t=u \left(t,m\left(\vartheta _0,t\right),\vartheta
_0,\varepsilon\right)$ is a solution of BSDE
\begin{align*}
{\rm d}Z_t=-F\left(t,m\left(\vartheta
_0,t\right),Z_t,s\left(t\right)\right){\rm d}t+s\left(t\right){\rm
  d}\bar W_t,\qquad Z_T=\Phi \left(m\left(\vartheta _0,T\right)\right),
\end{align*}
where $s\left(t\right)=B_\varepsilon \left(\vartheta _0,t\right)u'
\left(t,m\left(\vartheta _0,t\right),\vartheta _0,\varepsilon \right)$. As $u
\left(\cdot \right)\rightarrow U \left(\cdot \right)$ the corresponding
limit BSDE is
\begin{align*}
{\rm d}Z_t=-F\left(t,Y_t,Z_t,s_t\right){\rm d}t+s_t{\rm
  d}V_t,\qquad Z_T=\Phi \left(Y_T\right),
\end{align*}
where $Z_t=U\left(t,Y_t,\vartheta _9\right), s_t=b\left(\vartheta
_0,t\right)U_y'\left(t,Y_t,\vartheta _0 \right) $. 

We do not set $m\left(\vartheta _{t,\varepsilon }^\star,t\right) $ and $\hat
Z_t=u \left(t,m\left(\vartheta _{t,\varepsilon }^\star,t\right),\vartheta
_{t,\varepsilon }^\star,\varepsilon\right) $ since in this case we need
to solve the equations \eqref{2-6a} for many values of $\vartheta $ and
the relevant computational cost is very high. Introduce the recurrent equation
\begin{align}
\label{3-4}
{\rm d}\hat m _t=-q_\varepsilon \left(\vartheta _{t,\varepsilon
}^\star,t\right)\hat m _t\,{\rm d}t+\varepsilon ^{-1} A_\varepsilon \left(\vartheta
_{t,\varepsilon }^\star,t\right) {\rm d}X_t ,\quad \tau < t\leq T
\end{align}
where the initial value is  $\hat m _\tau
=m(\hat\vartheta _{\tau ,\varepsilon },\tau )  $.

Let us set
\begin{align*}
\hat Z_t=u \left(t,\hat m _t,\vartheta _{t,\varepsilon }^\star,\varepsilon\right),\qquad \hat
s_t=B_\varepsilon \left(\vartheta _{t,\varepsilon
}^\star,t\right)u_y'\left(t,\hat m _t,\vartheta _{t,\varepsilon }^\star,\varepsilon\right) 
\end{align*}

The main result of this work is the following theorem.
\begin{theorem}
\label{T1} Let the conditions  ${\cal A}$,${\cal B}$,${\cal C}$  be fulfilled. Then 
\begin{align}
\label{3-5}
\frac{\hat Z_t-Z_t}{\sqrt{\varepsilon }}&\Longrightarrow U'_y\left(t,Y_t,\vartheta
_0\right) \sqrt{\frac{b\left(\vartheta 
  _0,t\right)\sigma \left(t\right)}{2f\left(t\right)}} \,\left[\hat\zeta
_{t }-\hat\xi _{t
}\right]\, \nonumber \\
&\qquad \qquad +\frac{\dot  U\left(t,Y_t,\vartheta _0\right)}{{\rm I}_\tau
  ^t\left(\vartheta _0\right) } \;\int_{\tau 
}^{t}\sqrt{\frac{{}\;\dot b(\vartheta _0,s)^2f\left(s\right)}{{2b(\vartheta _0,s)  
    \sigma \left(s\right)}}}\,\, {\rm d}w\left(s\right). 
\end{align}
Here $\hat\zeta _{t }\sim {\cal N}\left(0,1\right),\hat\xi _{t }\sim {\cal
  N}\left(0,1\right) $ are mutually independent random variables and
  $w\left(s\right),\tau \leq s\leq T$ is a Wiener process. 

\end{theorem}
\begin{proof} First we notice
\begin{align}
\label{3-5a}
\hat Z_t- Z_t&=u \left(t,\hat m _t,\vartheta _{t,\varepsilon }^\star,\varepsilon\right)-U
\left(t,Y_t,\vartheta _0\right)\nonumber\\
&=u \left(t,\hat m _t,\vartheta _{t,\varepsilon }^\star,\varepsilon\right)- u
\left(t,m\left(\vartheta _0,t\right),\vartheta _{t,\varepsilon
}^\star,\varepsilon\right)\nonumber\\ 
&\qquad \quad +
 u
\left(t,m\left(\vartheta _0,t\right),\vartheta _{t,\varepsilon
}^\star,\varepsilon\right)-u 
\left(t,m\left(\vartheta _0,t\right),\vartheta _0,\varepsilon\right)\nonumber\\
&\qquad \quad +  u 
\left(t,m\left(\vartheta _0,t\right),\vartheta _0,\varepsilon\right)-u 
\left(t,Y_t,\vartheta _0,\varepsilon\right)\nonumber\\
&\qquad \quad + u 
\left(t,Y_t,\vartheta _0,\varepsilon\right)-u
\left(t,Y_t,\vartheta _0,0\right)\nonumber\\
&=u_y' \left(t,\tilde m_t,\vartheta _{t,\varepsilon }^\star,\varepsilon
\right)\left(\hat m _t-m\left(\vartheta _0,t\right) \right)\nonumber\\
&\qquad \quad +u_y' \left(t,\bar m_t,\vartheta _{t,\varepsilon }^\star,\varepsilon
\right)\left(m\left(\vartheta _0,t\right)-Y_t \right)\nonumber\\
&\qquad \quad +
\dot u(t,m\left(\vartheta _0,t\right),\tilde \vartheta ,\varepsilon)
\left(\vartheta _{t,\varepsilon }^\star-\vartheta _0 \right) +u'_\varepsilon
(t,Y_t,\vartheta ,\tilde \varepsilon)\varepsilon  . 
\end{align}
Here $\tilde m_t,\bar m_t,\tilde \vartheta $ are some intermediate points in
the corresponding expansions. We further study the quantities $\delta _t=\hat m
_t-m\left(\vartheta _0,t\right) $ and $m\left(\vartheta _0,t\right)-Y
_t $. Recall the equations for $m\left(\vartheta _0,t\right)$ and $\hat m_t$
\begin{align*}
&{\rm d}m\left(\vartheta _0,t\right)=-a\left(t\right)m\left(\vartheta
  _0,t\right){\rm d}t+\frac{\gamma 
    _*\left(\vartheta _0,t\right)f\left(t\right)}{ \sigma
    \left(t\right)}{\rm d}\bar W_t,\quad m\left(\vartheta
  _0,0\right)=0,\\
 &{\rm d}\hat m_t=-a\left(t\right)\hat m_t{\rm d}t-\frac{\gamma
    _*\left(\vartheta_{t,\varepsilon
    }^\star,t\right)f\left(t\right)^2}{\varepsilon \sigma
    \left(t\right)^2}\;\delta _t\;{\rm d}t+\frac{\gamma
    _*\left(\vartheta_{t,\varepsilon
    }^\star,t\right)f\left(t\right)}{ \sigma \left(t\right)}{\rm
    d}\bar W_t, 
\end{align*}
where $\hat m_\tau = m\left(\vartheta _{t,\varepsilon
  }^\star,\tau \right)$.
 Therefore for $\delta _t$ we obtain the equation
\begin{align*}
{\rm d}\delta _t&=-q_\varepsilon  \left(\vartheta _{t,\varepsilon }^\star,t\right)\delta
_t{\rm d}t +\frac{\left[\gamma
    _*\left(\vartheta_{t,\varepsilon
    }^\star,t\right)-\gamma
    _*\left(\vartheta_0,t\right)\right]f\left(t\right)^2}{ \sigma
    \left(t\right)^2}{\rm d}\bar W_t,\qquad
 \tau <t\leq T ,
\end{align*}
where $\delta _\tau =m(\hat\vartheta _{\tau ,\varepsilon
  }^\star,\tau )-m\left(\vartheta
_0,\tau \right)$  and
\begin{align*}
q_\varepsilon \left(\vartheta _{t,\varepsilon
}^\star,t\right)=a\left(t\right)+\frac{\gamma _*\left(\vartheta_{t,\varepsilon
  }^\star,t\right)f\left(t\right)^2}{\varepsilon \sigma \left(t\right)^2}.
\end{align*}

The solution of this equation on the time interval  $\left[\tau ,T\right]$ is 
\begin{align*} 
\delta _t&=\delta _\tau  e^{-\int_{\tau}^{t}q_\varepsilon \left(\vartheta
  _{v,\varepsilon }^\star ,v\right){\rm
    d}v}\\
&\qquad  +e^{-\int_{\tau}^{t}q_\varepsilon \left(\vartheta
  _{v,\varepsilon }^\star ,v\right){\rm
    d}v}\int_{\tau}^{t}e^{\int_{\tau}^{s}q_\varepsilon \left(\vartheta
  _{v,\varepsilon }^\star ,v\right){\rm d}v}\frac{\left[\gamma
    _*\left(\vartheta_{s,\varepsilon }^\star,s\right)-\gamma
    _*\left(\vartheta_0,s\right)\right]f\left(s\right)^2}{ \sigma
  \left(s\right)^2}{\rm d}\bar W_s.
\end{align*}
Note that at the vicinity of the point $t$ we have the expansion
\begin{align*}
q_\varepsilon  \left(\vartheta _{s,\varepsilon
}^\star,s\right)=\frac{1}{\varepsilon }\frac{\gamma 
    _*\left(\vartheta_0,t\right)f\left(t\right)^2}{ \sigma
    \left(t\right)^2}\left(1+ O\left(\varepsilon \right)+
O\left(s-t\right)+O\left(\sqrt{\varepsilon }   \right)
  \right), 
\end{align*}
where the relation $\vartheta _{s,\varepsilon }^\star-\vartheta _0
=O\left(\sqrt{\varepsilon }\right) $ is used. Let us denote $K\left(\vartheta
,t\right)={\gamma _*\left(\vartheta_0,t\right)f\left(t\right)^2}{ \sigma
  \left(t\right)^{-2}}$ and notice that
\begin{align*}
q_\varepsilon \left(\vartheta _{v,\varepsilon }^\star ,v\right)=q_\varepsilon
\left(\vartheta _0 ,v\right)\left(1 +O\left({\sqrt{\varepsilon }}\right)\right).
\end{align*}
 The same arguments as in the proof of Lemma 2 in
\cite{Kut19} for the stochastic integral lead us to
\begin{align*}
&e^{-\int_{\tau}^{t}q_\varepsilon \left(\vartheta _{v,\varepsilon }^\star ,v\right){\rm
    d}v}\int_{\tau}^{t}e^{\int_{\tau}^{s}q_\varepsilon \left(\vartheta _{v,\varepsilon }^\star ,v\right){\rm
    d}v}\frac{\left[\gamma
    _*\left(\vartheta_{s,\varepsilon
    }^\star,s\right)-\gamma
    _*\left(\vartheta_0,s\right)\right]f\left(s\right)^2}{ \sigma
    \left(s\right)^2}{\rm d}\bar W_s\\
&\quad =\int_{\tau}^{t}e^{-\int_{s}^{t}q_\varepsilon \left(\vartheta _0 ,v\right){\rm
    d}v}\frac{\dot\gamma
    _*\left(\vartheta_0,s\right)\left[\vartheta_{s,\varepsilon
    }^\star-\vartheta_0\right]f\left(s\right)^2}{ \sigma
    \left(s\right)^2}{\rm d}\bar W_s\left(1+o\left(1\right)\right)\\
&\quad =\int_{\tau}^{t}e^{-\frac{1}{\varepsilon }K\left(\vartheta
    _0,t\right)\left(t-s\right)}\frac{\dot\gamma 
    _*\left(\vartheta_0,s\right)\;\eta_{s,\varepsilon 
    }^\star\; f\left(s\right)^2}{ \sigma
    \left(s\right)^2}{\rm d}\bar W_s\sqrt{\varepsilon
  }\left(1+o\left(1\right)\right)\\ 
&\quad =   \frac{\dot\gamma
    _*\left(\vartheta_0,t\right)\;\eta_{t,\varepsilon
    }^\star\; f\left(t\right)^2}{ \sigma
    \left(t\right)^2 \sqrt{2K\left(\vartheta _0,t\right)} } \;\zeta
  _{t,\varepsilon } \;\varepsilon   \left(1+o\left(1\right)\right)\\ 
&\quad=   \frac{\dot\gamma
    _*\left(\vartheta_0,t\right)\;\eta_{t,\varepsilon
    }^\star\; f\left(t\right)}{ \sigma
    \left(t\right) \sqrt{2\gamma _*\left(\vartheta _0,t\right)} } \;\zeta
  _{t,\varepsilon }\; \varepsilon   \left(1+o\left(1\right)\right) . 
\end{align*}
Here $\zeta _{t,\varepsilon }\sim {\cal N}\left(0,1\right), t\in (\tau ,T]$
  are independent random variables. We remind that by Lemma \ref{L1} we have
\begin{align*}
\gamma _*\left(\vartheta _0,t\right)\longrightarrow \frac{b\left(\vartheta
  _0,t\right)\sigma \left(t\right)}{f\left(t\right)} ,\qquad \dot \gamma
_*\left(\vartheta _0,t\right)\longrightarrow \frac{\dot b\left(\vartheta
  _0,t\right)\sigma \left(t\right)}{f\left(t\right)}.
\end{align*}
The second limit here could be obtained similarly to the first one in Lemma 2
\cite{Kut19}. 

For the initial value we have 
\begin{align*}
\delta _\tau &=\dot m(\tilde\vartheta,\tau )(\hat\vartheta _{\tau ,\varepsilon
} -\vartheta _0)e^{-\int_{\tau}^{t}q_\varepsilon \left(\vartheta _{v,\varepsilon }^\star ,v\right){\rm
    d}v}\\
&=\dot m\left(\vartheta _0,\tau \right)(\hat\vartheta _{\tau
  ,\varepsilon } -\vartheta _0)\;
e^{-\int_{\tau}^{t}q_\varepsilon \left(\vartheta
  _0 ,v\right){\rm     d}v}\left(1+o\left(1\right)\right)\\ &=\dot
m\left(\vartheta _0,\tau \right) \hat\eta _{\tau ,\varepsilon
}\;e^{-\frac{1}{\varepsilon }\int_{\tau}^{t}K \left(\vartheta 
  _0 ,v\right){\rm
    d}v}\;\sqrt{\varepsilon}\;\left(1+o\left(1\right)\right)=O\left(e^{-\frac{c_*}{\varepsilon
  }\left(t-\tau \right)}\right),
\end{align*}
where $c_*=\inf_{\tau <v\leq \tau }K \left(\vartheta 
  _0 ,v\right) $.

Finally, we obtain the representation
\begin{align}
\label{3-6}
\hat m_t-m\left(\vartheta _0,t\right)=\sqrt{\frac{\dot b\left(\vartheta
    _0,t\right)^2f\left(t\right)}{2b\left(\vartheta
    _0,t\right)\sigma \left(t\right) }} \;\eta_{t,\varepsilon 
    }^\star \;\zeta
  _{t,\varepsilon }\; \varepsilon   \left(1+o\left(1\right)\right).
\end{align}

For the difference $m\left(\vartheta _0,t\right)-Y_t$ we have the
representation \eqref{2-6e}
\begin{align*}
m\left(\vartheta _0,t\right)-Y_t&=\int_{0}^{t}e^{-\frac{1}{\varepsilon
  }K\left(\vartheta _0,t\right) \left(t-s\right)} \frac{\gamma
  _*\left(\vartheta _0,s\right)f\left(s\right)}{\sigma \left(s\right)}{\rm
  d}W_s\left(1+o\left(1\right)\right)\\
&\qquad -\int_{0}^{t}e^{-\frac{1}{\varepsilon
  }K\left(\vartheta _0,t\right) \left(t-s\right)} b\left(\vartheta _0,s\right){\rm
  d}V_s\left(1+o\left(1\right)\right)\\
&=\frac{\gamma
  _*\left(\vartheta _0,t\right)f\left(t\right)}{\sigma \left(t\right)}
\int_{0}^{t}e^{-\frac{1}{\varepsilon
  }K\left(\vartheta _0,t\right) \left(t-s\right)} {\rm  
  d}W_s\left(1+o\left(1\right)\right)\\
&\qquad -b\left(\vartheta _0,t\right)\int_{0}^{t}e^{-\frac{1}{\varepsilon
  }K\left(\vartheta _0,t\right) \left(t-s\right)}{\rm
  d}V_s\left(1+o\left(1\right)\right)\\
&=\frac{\gamma
  _*\left(\vartheta _0,t\right)f\left(t\right)}{\sigma
  \left(t\right)\sqrt{2K\left(\vartheta _0,t\right)}} \,\hat\zeta
_{t,\varepsilon }\,  \sqrt{\varepsilon }\;\left(1+o\left(1\right)\right)\\
&\qquad  -\frac{b\left(\vartheta
  _0,t\right)}{\sqrt{2K\left(\vartheta _0,t\right)}} \;\hat\xi _{t,\varepsilon
}\,  \sqrt{\varepsilon }\;\left(1+o\left(1\right)\right) \\
&=\sqrt{\frac{b\left(\vartheta
  _0,t\right)\sigma \left(t\right)}{2f\left(t\right)}} \,\left[\hat\zeta
_{t,\varepsilon }-\hat\xi _{t,\varepsilon
}\right]\,  \sqrt{\varepsilon }\;\left(1+o\left(1\right)\right).
\end{align*}

From the convergences $m\left(\vartheta _0,t\right)\rightarrow Y_t $,
$\vartheta _{t,\varepsilon }^\star \rightarrow \vartheta _0 $ as $\varepsilon
\rightarrow 0$ and the continuity of derivatives we obtain the representation
\begin{align}
\label{3-7}
\frac{\hat Z_t-Z_t}{\sqrt{\varepsilon }}&=u'_y\left(t,Y_t,\vartheta
_0,0\right) \sqrt{\frac{b\left(\vartheta 
  _0,t\right)\sigma \left(t\right)}{2f\left(t\right)}} \,\left[\hat\zeta
_{t,\varepsilon }-\hat\xi _{t,\varepsilon
}\right]\,  \left(1+o\left(1\right)\right)\nonumber\\
&\qquad +\frac{\dot  u\left(t,Y_t,\vartheta _0,0\right)}{{\rm I}_\tau
  ^t\left(\vartheta _0\right) } \;\int_{\tau 
}^{t}\sqrt{\frac{{}\;\dot b(\vartheta _0,s)^2f\left(s\right)}{{2b(\vartheta _0,s)  
    \sigma \left(s\right)}}}\,\xi _{s,\varepsilon} {\rm d}\bar W_s
\;\left(1+o\left(1\right)\right). 
\end{align}
Therefore 
\begin{align*}
\frac{\hat Z_t-Z_t}{\sqrt{\varepsilon }}&\Longrightarrow u'_y\left(t,Y_t,\vartheta
_0,0\right) \sqrt{\frac{b\left(\vartheta 
  _0,t\right)\sigma \left(t\right)}{2f\left(t\right)}} \,\left[\hat\zeta
_{t }-\hat\xi _{t
}\right]\,  \\
&\qquad +\frac{\dot  u\left(t,Y_t,\vartheta _0,0\right)}{{\rm I}_\tau
  ^t\left(\vartheta _0\right) } \;\int_{\tau 
}^{t}\sqrt{\frac{{}\;\dot b(\vartheta _0,s)^2f\left(s\right)}{{2b(\vartheta _0,s)  
    \sigma \left(s\right)}}}\,\, {\rm d}w\left(s\right). 
\end{align*}
Here $\hat\zeta _{t }\sim {\cal N}\left(0,1\right),\hat\xi _{t }\sim {\cal
  N}\left(0,1\right) $ are mutually independent random variables and
$w\left(s\right),0\leq s\leq T$ is the Wiener process.

\end{proof}

Let us introduce the random  process
\begin{align*}
z\left(t,\vartheta _0,Y_t\right)=\frac{\dot  u\left(t,Y_t,\vartheta _0,0\right)}{{\rm I}_\tau
  ^t\left(\vartheta _0\right) } \;\int_{\tau 
}^{t}\sqrt{\frac{{}\;\dot b(\vartheta _0,s)^2f\left(s\right)}{{2b(\vartheta _0,s)  
    \sigma \left(s\right)}}}\,\, {\rm d}w\left(s\right)
\end{align*}
and notice that the Gaussian process $Y_t,0\leq t\leq T$ and Wiener process
$w\left(t\right),0\leq t\leq T$ are independent.  

We will further introduce a new condition.
\bigskip

${\cal D}.$ {\it The derivatives $u'_y\left(\cdot \right),\dot u\left(\cdot
  \right),u'_\varepsilon \left(\cdot \right)$ have polynomial majorants in $y$.}

\begin{corollary}
\label{C1} Let the conditions ${\cal A},{\cal B},{\cal C},{\cal D}$ be
fulfilled. Then for any
continuous functions $h\left(\cdot \right)$  we have
the relation
\begin{align}
\label{3-8}
\varepsilon ^{-1/2}\int_{\tau }^{T}h\left(t\right)\left[\hat Z_t-Z_t\right]{\rm
  d}t\Longrightarrow \int_{\tau }^{T}h\left(t\right)z\left(s,\vartheta _0,Y_s\right){\rm d}s .
\end{align}
\end{corollary}
\begin{proof} 
The proof follows from the limits
\begin{align*}
&\int_{\tau }^{T}r\left(Y_s,s\right)\hat\zeta _{s,\varepsilon }{\rm
    d}s\longrightarrow 0,\qquad 
&\int_{\tau }^{T}g\left(Y_s,s\right)\hat\xi_{s,\varepsilon }{\rm
    d}s\longrightarrow 0
\end{align*}
for any continuous functions $r\left(\cdot \right),g\left(\cdot \right)$ with
finite  moments (see \cite{Kut19}). 

Indeed for the process
\begin{align*}
\hat\zeta _{t,\varepsilon }=\sqrt{\frac{2K\left(\vartheta
            _0,t\right)}{\varepsilon }}\int_{0}^{t}e^{-\frac{1}{\varepsilon }K\left(\vartheta
            _0,t\right)\left(t-s\right)}
          {\rm d} W_{s}
\end{align*}
we have 
\begin{align*}
\Ex_{\vartheta _0} \hat\zeta _{t_1,\varepsilon }\hat\zeta _{t_2,\varepsilon
}=\sqrt{\frac{K_1K_2}{K_1+K_2}} \left[e^{-\frac{1}{\varepsilon
    }K\left|t_1-t_2\right|}-e^{-\frac{1}{\varepsilon
    }\left[K_1t_1+K_2t_2\right]}\right] \longrightarrow 0,
\end{align*}
where $K_i=K\left(\vartheta _0,t_i\right),i=1,2$ and set
$K=K_1\1_{\left\{t_1>t_2\right\}}+K_2\1_{\left\{t_1\leq t_2\right\}}$.

\end{proof}

\begin{corollary}
\label{C2} Let the conditions ${\cal A},{\cal B},{\cal C},{\cal D}$ be fulfilled. Then
\begin{align*}
\varepsilon ^{-1}\Ex_{\vartheta _0}\left(\hat Z_t-Z_t\right)^2&\longrightarrow
            {\frac{b\left(\vartheta _0,t\right)\sigma
                \left(t\right)}{f\left(t\right)}} \Ex_{\vartheta _0}
            U'_y\left(t,Y_t,\vartheta _0\right)^2\, +\frac{\Ex_{\vartheta _0}
              {\dot U\left(t,Y_t,\vartheta _0\right)^2}}{ {\rm I}_\tau
              ^t\left(\vartheta _0\right) }.
\end{align*}
\end{corollary}
\begin{proof} The proof follows from the same arguments as in Corollary \ref{C1}.

\end{proof}

\section{Discussion}

It is shown that from four  components of the error of approximation \eqref{3-5}
only two of them have main contribution (Theorem \ref{T1}). Moreover, if we
consider the integrated error, then we have just one term (Corollary
\ref{C1}). 

The contribution of the approximation of the conditional expectation
$m\left(\vartheta _0,t\right)$ by  the values $\hat m_t$ of solution of recurrent
equation (adaptive filtration) is negligible.  This means that for the observational model \eqref{2-1}, \eqref{2-2} with unknown parameter of volatility
function, the equation \eqref{3-4} proposing the approximation $\hat m_t$  of
$m\left(\vartheta _0,t\right)$ has error of order $\varepsilon $. This result
could be applied in filtration theory.

Several possible generalizations could be made quite
easily. For example, if we suppose  as in \cite{Kut19} that the function
$f\left(t\right)=f\left(\vartheta ,t\right)$, then any
construction  of approximation of $Z_t$ would be close to this one given here. 

Another statement of the problem could be obtained if $b\left(\vartheta
,t\right)$ is replaced by $\psi _\varepsilon b\left(\vartheta
,t\right)$, where $\psi _\varepsilon =\varepsilon ^\delta , \delta \in
(0,1/3]$. The problem of parameter estimation for such models was studied in
  \cite{Kut20b} and the problem of approximation of the solution of BSDE for
  such models could be considered as well.

{\bf Acknowledgment.} This research (sections 2.1-2.3) was financially
supported by the Ministry of Education and Science of the Russian Federation
(project no. FSWF-2020-0022) and the research (section 2.4) was
carried out with support by RSF project no. 20-61-47043.

\end{document}